\documentclass[12pt]{amsart}
 \usepackage[foot]{amsaddr}
\usepackage[top=0.9in, bottom=1in, left=1in, right=1in]{geometry}
\usepackage{amsmath, amssymb, amsthm}
\usepackage{verbatim}
\usepackage{graphicx}
\usepackage{enumerate}
\usepackage{mathrsfs}
\usepackage{tikz}
\usepackage{tikz-cd}
\usepackage{centernot}
\usepackage{hyperref}

\newtheorem{thm}{Theorem}[section]
\newtheorem{prop}[thm]{Proposition}
\newtheorem{lemma}[thm]{Lemma}
\newtheorem{cor}[thm]{Corollary}

\def\XXint#1#2#3{{\setbox0=\hbox{$#1{#2#3}{\int}$ }
\vcenter{\hbox{$#2#3$ }}\kern-.6\wd0}}

\newcommand{\C}{\mathbb{C}}

\theoremstyle{definition}
\newtheorem{definition}[thm]{Definition}
\newtheorem{example}[thm]{Example}

\theoremstyle{remark}

\numberwithin{equation}{section}

\newtheoremstyle{ser}
{8pt}
{8pt}
{\it}
{}
{\sf}
{:}
{6mm}
{}

\newtheoremstyle{serr}
{8pt}
{8pt}
{\normalfont}
{}
{\sf}
{.}
{6mm}
{}

\theoremstyle{ser}

\theoremstyle{serr}

\tikzset{node distance=2cm, auto}
\makeatletter
\newcommand*\bigcdot{\mathpalette\bigcdot@{.5}}
\newcommand*\bigcdot@[2]{\mathbin{\vcenter{\hbox{\scalebox{#2}{$\m@th#1\bullet$}}}}}
\makeatother

\begin{document}

\title{The Douglas Lemma for von Neumann Algebras and Some Applications}

\author{Soumyashant Nayak}
\address{Smilow Center for Translational Research\\
 University of Pennsylvania\\
  Philadelphia\\
   PA 19104\\
   ORCiD: 0000-0002-6643-6574}
\email{nsoum@pennmedicine.upenn.edu}
\urladdr{https://nsoum.github.io/} 

\begin{abstract}
In this article, we discuss some applications of the well-known Douglas factorization lemma in the context of von Neumann algebras. Let $\mathcal{B}(\mathscr{H})$ denote the set of bounded operators on a complex Hilbert space $\mathscr{H}$, and $\mathscr{R}$ be a von Neumann algebra acting on $\mathscr{H}$. We prove some new results about left (or, one-sided) ideals of von Neumann algebras; for instance, we show that every left ideal of $\mathscr{R}$ can be realized as the intersection of a left ideal of $\mathcal{B}(\mathscr{H})$ with $\mathscr{R}$. We also generalize a result by Loebl and Paulsen (Linear Algebra Appl. 35 (1981), 63--78) pertaining to $C^*$-convex subsets of $\mathcal{B}(\mathscr{H})$ to the context of $\mathscr{R}$-bimodules.

\vskip 0.01in \noindent
{\bf MSC2010 subject classification:} 16D25, 47L20, 52A01
\end{abstract}

\maketitle

\section{Introduction}
In \cite{douglas-factor}, Douglas notes that the notions of majorization, factorization, and range inclusion, for operators on a Hilbert space are intimately connected. We mention the main result  of \cite{douglas-factor} below which is referred to as the {\it Douglas lemma} or the {\it Douglas factorization theorem} in the literature. 
\begin{thm}[Douglas lemma]
{\sl For bounded operators $A, B$ on a Hilbert space $\mathscr{H}$, the following statements are equivalent :
\begin{itemize}
\item[(i)] (majorization) $A^*A \le \lambda ^2 B^*B$ for some $\lambda \ge 0$;
\item[(ii)] (factorization) $A = CB$ for some bounded operator $C$ on $\mathscr{H}$;
\item[(iii)] (range inclusion) range($A^*$) $\subseteq$ range($B^*$).
\end{itemize}
}
\end{thm}
It naturally appears in many contexts, and as Douglas observed, ``$\dots$ fragments of these results are to be found scattered throughout the literature (usually buried in proofs) $\dots$ ." In this article, we give a constructive proof of the Douglas lemma for von Neumann algebras. In \S \ref{sec:app}, we discuss some applications of this result to the structure of left (or, one-sided) ideals of von Neumann algebras, and to the notion of $C^*$-convexity in the context of bimodules over a von Neumann algebra. 

For the convenience of the reader, we briefly recall some basic notions in operator algebras and set up the notation. We shall denote a complex Hilbert space by $\mathscr{H}$ and the set of bounded operators on $\mathscr{H}$ by $\mathcal{B}(\mathscr{H})$. The family $\mathcal{B}(\mathscr{H})$ is an algebra relative to the usual addition and multiplication (composition) of operators.  Let $\| \cdot \|$ denote the usual operator norm. Provided with this norm, $\mathcal{B}(\mathscr{H})$  becomes a Banach algebra.  A family $\Gamma$ of operators on $\mathscr{H}$ is said to be ``self-adjoint'' when $A^*$, the adjoint-operator of $A$, is in $\Gamma$ if $A$ is in $\Gamma$. The norm-closed self-adjoint subalgebras of $\mathcal{B}(\mathscr{H})$ are called ``C*-algebras'' and those closed in the weak-operator topology on $\mathcal{B}(\mathscr{H})$ are the ``von Neumann algebras''.  Our von Neumann  algebras are required to contain the identity operator $I$ on $\mathscr{H}$, that is, $Ix=x$ for each $x$ in $\mathscr{H}$. We assume that our $C^*$-algebras are unital. We often denote a C*-algebra by ``$\mathfrak{A}$'' and a von Neumann algebra by ``$\mathscr{R}$.'' In the rest of this section, we discuss the two contexts in which we apply the Douglas lemma - (i) one-sided ideal structure of von Neumann algebras, (ii) $C^*$-convexity in bimodules over a $C^*$-algebra.

\subsection{Ideals of $C^*$-algebras and von Neumann algebras}
The structures of the left ideals and right ideals in a $C^*$-algebra, $\mathfrak{A}$, are very closely tied to the representation theory of $\mathfrak{A}$; for instance, via the GNS construction. In particular, the representation theory of $\mathcal{B}(\mathscr{H})$ is very much a part of this. It is well-known that the weak-operator closed left ideals in a von Neumann algebra $\mathscr{R}$ are left principal ideals of the form $\mathscr{R}E$ for a projection $E$ in $\mathscr{R}$. In Lemma \ref{lem:isolated} we prove that for a positive self-adjoint operator $A$ in $\mathscr{R}$, the left principal ideal $\mathscr{R}A$ is weak-operator closed if and only if $0$ is an isolated point in the spectrum of $A$. 

For the discussion in this paragraph, we assume that the von Neumann algebra $\mathscr{R}$ acting on $\mathscr{H}$ is infinite-dimensional, so as to avoid making vacuous statements. In Theorem \ref{thm:count-gen}, we show that a norm-closed left ideal in $\mathscr{R}$ which is not weak-operator closed must be (algebraically) generated by uncountably many operators in $\mathscr{R}$. We include a proof of the fact that the lattice of norm-closed left ideals of a $C^*$-algebra acting on $\mathscr{H}$ may be derived from the lattice of norm-closed left ideals of $\mathcal{B}(\mathscr{H})$, via intersection with the $C^*$-algebra. On a similar note, in Corollary \ref{cor:intersection} we show that the lattice of left ideals of $\mathscr{R}$ may be derived from the lattice of left ideals of $\mathcal{B}(\mathscr{H})$, via intersection with $\mathscr{R}$. Although it is straightforward to see that the intersection of a left ideal of $\mathcal{B}(\mathscr{H})$ with $\mathscr{R}$ is a left ideal of $\mathscr{R}$, what we show is that {\it every} left ideal of $\mathscr{R}$ can be obtained in such a manner.

\subsection{C*-convexity}
\label{subsec:c-star}
The numerical range of an operator $T$ in $\mathcal{B}(\mathscr{H})$ is defined as $$W(T) := \{ \langle Tx, x \rangle : x \in \mathscr{H}, \|x \| = 1 \}.$$ The Toeplitz-Hausdorff theorem (cf. \cite{toeplitz}, \cite{hausdorff}) states that the numerical range of a bounded operator is a convex set. This subset of the complex plane $\mathbb{C}$ succinctly captures information about the eigenvalues, algebraic, analytic structure
of $T$ in the geometry of its boundary. In \cite{arveson1}, \cite{arveson2}, Arveson defines the notion of the $n^{\mathrm{th}}$-matrix range (a generalized  non-commutative numerical range) of an operator $T$  as $$W_n(T) := \{ \Phi(T) : \Phi \textrm{ is a unital completely positive map from } C^*(T) \textrm{ to } M_n(\mathbb{C}) \}.\footnote{Note that $W_1(T)$ is not necessarily the same as the numerical range $W(T)$. But they have the same closure in $\mathbb{C}$.}$$

The description of completely positive maps given by Stinespring's theorem (cf. \cite{stinespring}), and Choi's theorem (cf. \cite{choi}) for completely positive maps between finite-dimensional $C^*$-algebras involving Kraus operators suggest the importance of studying a non-commutative version of convexity called $C^*$-convexity which we define below in the context of bimodules over a $C^*$-algebra. 

\begin{definition}
Let $\mathfrak{A}$ be a $C^*$-algebra with identity $I$ and $\mathfrak{H}$ be a $\mathfrak{A}$-bimodule. For vectors $A_1, \cdots, A_n$ in $\mathfrak{H}$ ($n \in \mathbb{N}$), and operators $T_1, \cdots, T_n \in \mathfrak{A}$ satisfying $T_1^* T_1 + \cdots + T_n^* T_n = I$, the vector $T_1^* A_1 T_1 + \cdots + T_n^* A_n T_n$ in $\mathfrak{H}$ is called a {\it $C^*$-convex} (or {\it $\mathfrak{A}$-convex}) {\it combination} of the $A_i$'s.

A subset $\mathscr{S}$ of $\mathfrak{H}$ is said to be \emph{$C^*$-convex} in $\mathfrak{H}$ (or \emph{$\mathfrak{A}$-convex}) if for vectors $A_1, \cdots, A_n$ in $\mathscr{S}$ ($n \in \mathbb{N}$), every $\mathfrak{A}$-convex combination of the $A_i$'s is also in $\mathscr{S}$.
\end{definition}

\begin{definition}
For vectors $A_1, A_2, \cdots, A_n$ in the $\mathfrak{A}$-bimodule $\mathfrak{H}$, the set $$\{ \sum_{i=1}^n T_i ^* A_i T_i : T_1, T_2, \cdots, T_n \in \mathfrak{A}, \sum_{i=1}^n T_i ^* T_i = I \} \subseteq \mathfrak{H}$$ is said to be the {\it $C^*$-polytope} (or {\it $\mathfrak{A}$-polytope}) generated by the $n$-tuple $\mathbf{A} := (A_1, A_2, \cdots, A_n)$. The $\mathfrak{A}$-polytope generated by a $2$-tuple $(A_1, A_2)$ is called the {\it $C^*$-segment} (or {\it $\mathfrak{A}$-segment}) joining the elements $A_1, A_2$ in $\mathfrak{H}$. 
\end{definition}

We will use the terms ``$C^*$-convex set", ``$C^*$-polytope'', ``$C^*$-segment" in the context of a bimodule over a general $C^*$-algebra. In the setting of a specific $C^*$-algebra $\mathfrak{A}$, we prefer to use the terms ``$\mathfrak{A}$-convex set", ``$\mathfrak{A}$-polytope", ``$\mathfrak{A}$-segment".

\begin{example}
\label{ex:C_star}
\begin{itemize}
\item[(i)] A $C^*$-algebra may be viewed as a $\mathbb{C}$-bimodule with the left and right action both given by the usual scaling. In this case, $C^*$-convexity reduces to the usual notion of convexity.

\item[(ii)] A $C^*$-algebra $\mathfrak{A}$ may be viewed as a $\mathfrak{A}$-bimodule with the left action (right action, respectively) given by multiplication on the left (on the right, respectively). For $\mathfrak{A} = \mathcal{B}(\mathscr{H})$ this is the context in which $C^*$-convexity is discussed by Loebl and Paulsen in \cite{loebl-paulsen}.
\end{itemize}

\end{example}

\begin{example}
Consider $M_n(\C)$ as an $M_n(\C)$-bimodule in the sense of Example \ref{ex:C_star}, (ii). The $n^{\mathrm{th}}$-matricial ranges $W_n(T) \subset M_n(\mathbb{C})$ are not only convex but also $M_n(\mathbb{C})$-convex. 
\end{example}

At this point, we direct the interested reader to \cite{loebl-paulsen} for an exposition on some basic results in the theory of $C^*$-convexity. Line segments in complex (or real) vector spaces are the most basic of convex sets. But although every $C^*$-convex set is convex, as a consequence of the non-commutativity of the ``coefficients", the $C^*$-segments in bimodules over a $C^*$-algebra need not be $C^*$-convex or even convex. For instance, in $M_2(\mathbb{C})$  (viewed as an $M_2(\mathbb{C})$-bimodule) consider 
$$A = \begin{bmatrix}
1 & 0\\
0 & 0
\end{bmatrix}, B = \begin{bmatrix}
0 & 0\\
0 & 0
\end{bmatrix}, C= \begin{bmatrix}
0 & 0\\
0 & 1
\end{bmatrix}.$$
As $A$ is unitarily equivalent to $C$, the $M_2(\mathbb{C})$-segment $S(A, B)$ contains $C$. The mean of $A$ and $C$ is of full-rank but the elements in $S(A, B)$ have rank $0$ or $1$. Hence $\frac{A+C}{2}$ is not in $S(A, B)$ although both $A, C$ are in $S(A, B)$.

A subset $S$ of a complex vector space $V$ is said to be convex if it contains the line segment joining any two points in it. Equivalently, the subset $S$ is said to be convex if it contains all convex combinations of its elements.
\begin{definition}
\label{def:pseudo-cstar}
Let $\mathfrak{A}$ be a unital $C^*$-algebra. We say that a set $\mathscr{S}$ in a $\mathfrak{A}$-bimodule $\mathfrak{H}$ is \emph{$C^*$-$2$-convex} in $\mathfrak{H}$ (or {\it $\mathfrak{A}$-$2$-convex}) if the $\mathfrak{A}$-segment joining any two elements in $\mathscr{S}$ is contained in $\mathscr{S}$.
\end{definition}

For complex vector spaces (thought of as $\C$-bimodules in the sense of Example \ref{ex:C_star}, (i)), the notions of $C^*$-convexity and $C^*$-$2$-convexity coincide.  Let $\mathfrak{A}$ be a $C^*$-algebra. In a $\mathfrak{A}$-bimodule $\mathfrak{H}$, clearly a $\mathfrak{A}$-convex set is $\mathfrak{A}$-$2$-convex. Because of the (generally) non-convex nature of $C^*$-segments, it is not readily apparent whether every $\mathfrak{A}$-$2$-convex set is also $\mathfrak{A}$-convex. As determining $C^*$-$2$-convexity of a set is a more direct affair involving pairs of elements, a result in the affirmative would be of practical utility in determining $C^*$-convexity of subsets of $\mathfrak{H}$.

The results in \cite[Theorem 15, 16]{loebl-paulsen} proved by Loebl and Paulsen may be more generally viewed as having shown that $\mathcal{B}(\mathscr{H})$-convexity and $\mathcal{B}(\mathscr{H})$-$2$-convexity are equivalent concepts in the context of $\mathcal{B}(\mathscr{H})$-bimodules. In Theorem \ref{prop:segment}, for a finite von Neumann algebra $\mathscr{R}$, we show that a subset $\mathscr{S}$ of a $\mathscr{R}$-bimodule is $\mathscr{R}$-convex if and only if $\mathscr{S}$ is $\mathscr{R}$-$2$-convex. In Theorem \ref{prop:segment2}, we prove a similar equivalence for properly infinite von Neumann algebras. Using the type decomposition of von Neumann algebras, we obtain the result for any von Neumann algebra. 

\subsection{Acknowledgments}
This article is based on a portion of the author's doctoral dissertation submitted to the University of Pennsylvania, Philadelphia. The general references used are \cite{kadison-ringrose1}, \cite{kadison-ringrose2}. The author acknowledges and extends his heartfelt gratitude to his former advisor, Prof. Richard V. Kadison, for many stimulating conversations about mathematics in general, and operator algebras in particular.

\section{The Douglas lemma for von Neumann algebras}
\label{sec:douglas}
\begin{thm}[Douglas factorization lemma]
\label{theorem:douglas}
\textsl{
Let $\mathscr{R}$ be a von Neumann algebra acting on the Hilbert space $\mathscr{H}$. For $A, B$ in $\mathscr{R}$ the following are equivalent :
\begin{itemize}
\item[(i)] $A^*A \le \lambda^2 B^*B$ for some $\lambda \ge 0$;
\item[(ii)] $A=CB$ for some operator $C$ in $\mathscr{R}$.
\end{itemize}
In addition, if $A^*A = B^*B$, then $C$ can be chosen to be a partial isometry with initial projection the range projection of $B$, and final projection as the range projection of $A$.
}
\end{thm}
\begin{proof}
(i) $\Longrightarrow$ (ii)\\
For any vector $x$ in the Hilbert space $\mathscr{H}$, we have that $\|Ax\|^2 = \langle A^*A x, x \rangle \le \lambda^2 \langle B^*B x, x\rangle = \lambda^2 \|Bx\|^2$ which implies $\|Ax\| \le \lambda \|Bx\|$. Thus if $Bx=0$, it follows that $Ax = 0$ and the linear map $C$ defined on the range of $B$ by $C(Bx) = Ax$ is well-defined and also bounded (with norm less than or equal to $\lambda$). Thus we may extend the domain of definition of $C$ to $\mathrm{ran}(B)^{-}$ the closure of the range of $B$. If $z$ is a vector in $\mathrm{ran}(B)^{\perp}$, we define $Cz = 0$. Thus $C$ is a bounded operator on $\mathscr{H}$ such that $A=CB$ with $\| C \| \le \lambda$.

Let $R$ be a self-adjoint operator in the commutant $\mathscr{R}'$ of $\mathscr{R}$. Then $RA = AR, RB = BR$ and the linear subspace $\mathrm{ran}(B)$ is invariant under $R$ and so is the closed subspace $\mathrm{ran}(B)^{\perp}$ (as $R$ is self-adjoint). For vectors $x_1$ in $\mathscr{H}$ and $x_2$ in $\mathrm{ran}(B)^{\perp}$, we have that $CR(Bx_1 + x_2) = CRBx_1 + C(Rx_2) = CB(Rx_1) + 0 = A(Rx_1) = R(Ax_1) = RCBx_1 = RC(Bx_1 + x_2)$. Thus $RC$ and $CR$ coincide on the dense subspace of $\mathscr{H}$ given by $\mathrm{ran}(B) \oplus \mathrm{ran}(B)^{\perp}$. Being bounded operators, we note that $RC = CR$ for any self-adjoint operator $R$ in $\mathscr{R}'$. As every element in a von Neumann algebra can be written as a finite linear combination of self-adjoint elements, we conclude that $C$ commutes with every element in $\mathscr{R}'$. By the double commutant theorem (cf. \cite{von-dc}), $C$ is in $(\mathscr{R}')' = \mathscr{R}$.

(ii) $\Longrightarrow$ (i)\\
If $A=CB$ for some operator $C \in \mathscr{R}$, then $A^*A = B^*C^*CB \le \|C\|^2 B^*B$. Thus, we may pick $\lambda = \|C\|$.

If $A^*A = B^*B$, then $\|Af\| = \|Bf\|$ for any vector $f$ in $\mathscr{H}$. Thus, the second part follows from the explicit definition of the operator $C$ earlier in the proof.
\end{proof}

The polar decomposition theorem for von Neumann algebras is a direct consequence of the Douglas lemma.
\begin{cor}[Polar decomposition theorem]
\textsl{
Let $\mathscr{R}$ be a von Neumann algebra acting on the Hilbert space $\mathscr{H}$. For an operator $A$ in $\mathscr{R}$, there is a partial isometry $V$ with initial projection the range projection of $(A^*A)^{\frac{1}{2}}$, and final projection as the range projection of $A$ such that $A = V(A^*A)^{\frac{1}{2}}$.
}
\end{cor}
\begin{proof}
Let $B$ denote the operator $(A^*A)^{\frac{1}{2}}$. Clearly $A^*A = B^*B$ and thus from the second part of the Douglas lemma, the corollary follows.
\end{proof}

\section{Applications}
\label{sec:app}

\subsection{Left ideals of von Neumann algebras}

\begin{lemma}
\label{lem:left ideal}
\textsl{
Let $A, B$ be operators in a von Neumann algebra $\mathscr{R}$. The left ideal $\mathscr{R}A$ in $\mathscr{R}$ is contained in the left ideal $\mathscr{R}B$ if and only if $A^*A \le \lambda^2 B^*B$ for some $\lambda \ge 0$. As a consequence, for any $A$ in $\mathscr{R}$, we have that $\mathscr{R}A=\mathscr{R}\sqrt{A^*A}.$
}
\end{lemma}
\begin{proof}
Its straightforward to see that $A$ is in $\mathscr{R}B$ if and only if $\mathscr{R}A \subseteq \mathscr{R}B$. And from Theorem \ref{theorem:douglas}, we have that $A$ is in $\mathscr{R}B$ if and only if $A^*A \le \lambda^2 B^*B$ for some $\lambda \ge 0$. 

Further, $\mathscr{R}A = \mathscr{R}B$ if and only if $B^*B \le \lambda^2 A^*A$ and $A^*A \le \mu^2 B^*B$ for some $\lambda, \mu \ge 0$. In particular, if $A^*A=B^*B$, then $\mathscr{R}A = \mathscr{R}B$. Noting that $A^*A = \sqrt{A^*A} \sqrt{A^*A}$, we conclude that $\mathscr{R}A = \mathscr{R} \sqrt{A^*A}$.
\end{proof}

\begin{lemma}
\label{lem:isolated}
\textsl{
Let $A$ be an operator in a von Neumann algebra $\mathscr{R}$. Then the left ideal $\mathscr{R}A$ is weak-operator closed if and only if $0$ is an isolated point in the spectrum of $A^*A$.
}
\end{lemma}
\begin{proof}
If $A = 0$, the conclusion is straightforward. So we may assume that $A \ne 0$.

If $\mathscr{R}A$ is weak-operator closed, there is a unique projection $E$ in $\mathscr{R}$ such that $\mathscr{R}A = \mathscr{R}E$. From Lemma \ref{lem:left ideal}, there are $\mu, \lambda > 0$ such that $\mu^2 E \le A^*A \le \lambda^2 E$. This tells us that the spectrum of $A^*A$ is contained in $\{0\} \cup [\mu, \lambda]$ which implies that $0$ is an isolated point in the spectrum of $A^*A$.

For the converse, let $0$ be an isolated point in the spectrum of $A^*A$. By the spectral mapping theorem, $0$ is also an isolated point in the spectrum of $\sqrt{A^*A}$. Let the distance of $0$ from $\textrm{sp}(\sqrt{A^*A}) - \{0 \}$ (which is compact as $0$ is isolated) be $\mu > 0$ and $\lambda = \|A\|$. Let $F$ be the projection onto the kernel of $\sqrt{A^*A}$, which is the largest projection in $\mathscr{R}$ such that $\sqrt{A^*A}F = 0$. We have that $\mu^2 (I-F) \le A^*A \le \lambda^2 (I-F)$. Thus $\mathscr{R}A = \mathscr{R}(I-F)$ which is weak-operator closed.
\end{proof}

\begin{prop}
\label{prop:norm-weak}
\textsl{
Let $A$ be an operator in a von Neumann algebra $\mathscr{R}$ acting on the Hilbert space $\mathscr{H}$. Then the left ideal $\mathscr{R}A$ is norm-closed if and only if $\mathscr{R}A$ is weak-operator closed.
}
\end{prop}
\begin{proof}
Let $\mathscr{R}A$ be norm-closed. Without loss of generality, we may assume that $A$ is positive (since $\mathscr{R}A = \mathscr{R}
\sqrt{A^*A}$). By the Stone-Weierstrass theorem, for a continuous function $f$ on $\mathrm{sp}(A)$ vanishing at $0$, $f(A)$ is in $\mathscr{R}A$. In particular, $\sqrt{A}$ is in $\mathscr{R}A$. By Lemma \ref{lem:left ideal}, there is a $\lambda > 0$ such that $(\sqrt{A})^2 = A \le \lambda^2
A^2$. The operator $\lambda^2 A^2 - A$ is positive and by the spectral mapping theorem, $$\mathrm{sp}(\lambda ^2 A^2 - A) := \{ \lambda^2 \mu^2 - \mu : \mu \in \textrm{sp}(A) \}.$$ For a non-zero element $\mu$ in the spectrum of $A$, $\lambda^2 \mu^2 - \mu \ge 0 \Rightarrow \mu \ge \frac{1}{\lambda^2}$. This tells us that $0$ is an isolated point in the spectrum of $A$ and hence by Lemma \ref{lem:isolated}, $\mathscr{R}A$ is weak-operator closed.

The converse is straightforward as the weak-operator topology on $\mathscr{R}$ is coarser than the norm topology.
\end{proof}

Let $\mathscr{R}$ be a von Neumann algebra. We use the notation $\left\langle V \right\rangle$, to denote the linear span of a subset $V$ of $\mathscr{R}$. 

\begin{definition}
Let $\mathcal{S}$ be a family of operators in the von Neumann algebra $\mathscr{R}$. The smallest left ideal of $\mathscr{R}$ containing $\mathcal{S}$ is denoted by $\left\langle \mathscr{R}\mathcal{S} \right\rangle$ and said to be \textit{generated} by $\mathcal{S}$. A left ideal $\mathscr{I}$ is said to be {\it finitely generated} ({\it countably generated}, respectively) if $\mathscr{I} = \left\langle \mathscr{R}\mathcal{S} \right\rangle$ for a finite (countable, respectively) subset $\mathcal{S}$ of $\mathscr{R}$. Here we take a moment to stress that the set of generators is considered in a  purely algebraic sense.
\end{definition}

\begin{prop}
\label{prop:singly-gen}
\textsl{
Let $A_1, A_2$ be operators in a von Neumann algebra $\mathscr{R}$. Then $\mathscr{R}A_1 + \mathscr{R}A_2 = \mathscr{R} \sqrt{A_1^*A_1 + A_2^*A_2}$. Thus, every finitely generated left ideal of $\mathscr{R}$ is a principal ideal.
}
\end{prop}
\begin{proof}
Consider the operators $A, \widetilde{A}$ in $M_2(\mathscr{R})$ represented by,
\[
A=
  \begin{bmatrix}
    A_1 & 0 \\
    A_2 & 0
  \end{bmatrix},
  \widetilde{A} = 
   \begin{bmatrix}
    \sqrt{A_1^*A_1 + A_2^*A_2} & 0 \\
    0 & 0
  \end{bmatrix}
\]
It is easy to see that $A^*A = \widetilde{A}^* \widetilde{A}$. By Lemma \ref{lem:left ideal}, $M_2(\mathscr{R})A = M_2(\mathscr{R})\widetilde{A}$ and comparing the $(1,1)$ entry on both sides, our result follows. Inductively, for operators $A_1, \dotsc, A_n$ in $\mathscr{R}$, we see that $$\mathscr{R}A_1 + \dotsb + \mathscr{R}A_n = \mathscr{R} \sqrt{A_1^*A_1 + \dotsb + A_n^*A_n}.$$ In conclusion, every finitely generated left ideal of $\mathscr{R}$ is singly generated.
\end{proof}

\begin{cor}
\label{cor:norm-weak}
\textsl{
If $\mathscr{I}$ is a norm-closed left ideal of $\mathscr{R}$ which is finitely generated, then $\mathscr{I}$ is weak-operator closed.
}
\end{cor}
\begin{proof}
A straightforward consequence of Proposition \ref{prop:norm-weak}, \ref{prop:singly-gen}.
\end{proof}

\begin{thm}
\label{thm:count-gen}
\textsl{
If $\mathscr{I}$ is a norm-closed left ideal of $\mathscr{R}$ which is countably generated, then $\mathscr{I}$ is weak-operator closed (and thus, a left principal ideal).
}
\end{thm}
\begin{proof}
Let $\mathscr{I}$ be a countably generated norm-closed left ideal of $\mathscr{R}$ with generating set $\mathcal{S} := \{ A_i : i \in \mathbb{N} \}$. We prove that it must be weak-operator closed. Noting that $\mathscr{R}A_i = \mathscr{R}\sqrt{A_i ^* A_i}$ and after appropriate scaling, we may assume that the $A_i$'s are positive contractions (i.e. $A_i$'s are positive and $\|A_i \| \le 1$). For $n \in \mathbb{N}$, define $B_n := \sqrt{\sum_{i=1}^n \frac{A_i ^2}{2^n}}$. Thus the sequence $\{ B_i^2 \}_{i=1}^{\infty}$ is an increasing Cauchy sequence of positive operators in $\mathscr{I}$, and $\lim_{n \rightarrow \infty} B_n^2$ exists. As $\mathscr{I}$ is norm-closed, the positive operator $B := \sqrt{\lim_{n \rightarrow \infty} B_n^2}$ is in $\mathscr{I}$ and thus $\mathscr{R}B \subseteq \mathscr{I}$. Also for each $n \in \mathbb{N}$ as $A_n ^2 \le 2^n B_n^2 \le 2^n B^2$, by Lemma \ref{lem:left ideal}, we have that $\mathscr{R}A_n \subseteq \mathscr{R}B$. Thus $\mathscr{I} \subseteq \mathscr{R}B$ and combined with the previous conclusion, $\mathscr{I} = \mathscr{R}B$. By Corollary \ref{cor:norm-weak}, being norm-closed, $\mathscr{I} = \mathscr{R}B$ is also weak-operator closed.
\end{proof}

Below we note a result about norm-closed left ideals of represented C*-algebras. In the results that follow after, we will see how a similar conclusion holds for left ideals in von Neumann algebras.

\begin{prop}
\label{prop: ideal}
\textsl{
Let $\mathfrak{A}$ be a C*-algebra acting on the Hilbert space $\mathscr{H}$ and let $\mathscr{I}$ be a norm-closed left ideal of $\mathfrak{A}$. Then there is a norm-closed left ideal $\mathscr{J}$ of $\mathcal{B}(\mathscr{H})$ such that $\mathscr{I} = \mathscr{J} \cap \mathfrak{A}$.
}
\end{prop}

\begin{proof}
For a state $\rho$ on a $C^*$-algebra, we denote its left kernel, as defined in \cite[\S 4.5.2]{kadison-ringrose1}, by $\mathcal{L}_{\rho}$. Let $\mathscr{P}^I$ denote the set of pure states on $\mathfrak{A}$ whose left kernels contain
$\mathscr{I}$. Then from Theorem 3.2 in \cite{kadison-irreducible}, we have that $$\mathscr{I} = \bigcap_{\rho \in \mathscr{P}^I} \mathcal{L}_{\rho}$$
A pure state $\rho$ on $\mathfrak{A}$ can be extended to a pure state $\overline{\rho}$ on $\mathcal{B}(\mathscr{H})$. We denote the set of all such extensions of the states in $\mathscr{P}^I$ by $\overline{\mathscr{P}}^I$. Being an intersection of norm-closed left ideals, the set $$\mathscr{J} := \bigcap_{\overline{\rho} \in \overline{\mathscr{P}}^I} \mathcal{L}_{\overline{\rho}}$$ is also a norm-closed left ideal of $\mathcal{B}(\mathscr{H})$. Clearly if $\overline{\rho}$ in $\overline{\mathscr{P}}^I$ is an extension of a state $\rho$ in $\mathscr{P}^I$, we have that  $\mathcal{L}_{\overline{\rho}} \cap \mathfrak{A} = \mathcal{L}_{\rho}$. Thus we conclude that $\mathscr{I} = \mathscr{J} \cap \mathfrak{A}$.
\end{proof}

\begin{prop}
\textsl{
Let $\mathscr{R}_1, \mathscr{R}_2$ be von Neumann algebras acting on the Hilbert space $\mathscr{H}$. Let $A$ be an operator in $\mathscr{R}_1 \cap \mathscr{R}_2$. Then $\mathscr{R}_1 A \cap \mathscr{R}_2 = \mathscr{R}_1 A \cap \mathscr{R}_2 A = (\mathscr{R}_1 \cap \mathscr{R}_2)A$.
}
\end{prop}
\begin{proof}
Let $B$ be an operator in $\mathscr{R}_1 A \cap \mathscr{R}_2$. As $B \in \mathscr{R}_1 A$, we have that $B^*B \le \lambda ^2 A^*A$ for some $\lambda \ge 0$. As $B, A$ are both in $\mathscr{R}_2$, we conclude from Lemma \ref{lem:left ideal} that $B$ is also in $\mathscr{R}_2A$. Thus $B \in \mathscr{R}_1 A \cap \mathscr{R}_2 A$. This proves that $\mathscr{R}_1 A \cap \mathscr{R}_2 \subseteq \mathscr{R}_1 A \cap \mathscr{R}_2 A$. The reverse inclusion is obvious. Thus $\mathscr{R}_1 A \cap \mathscr{R}_2 = \mathscr{R}_1 A \cap \mathscr{R}_2 A$.

By considering the von Neumann algebra $\mathscr{R}_1 \cap \mathscr{R}_2$ (in place of $\mathscr{R}_2$), we have from the above that $\mathscr{R}_1 A \cap \mathscr{R}_2 = \mathscr{R}_1 A \cap (\mathscr{R}_1 \cap \mathscr{R}_2) = \mathscr{R}_1 A \cap (\mathscr{R}_1 \cap \mathscr{R}_2)A = (\mathscr{R}_1 \cap \mathscr{R}_2)A$.
\end{proof}

\begin{cor}
\label{cor:intersection}
\textsl{
Let $\mathscr{R}_1, \mathscr{R}_2$ be von Neumann algebras acting on the Hilbert space $\mathscr{H}$. Let $\mathcal{S}$ be a family of operators in $\mathscr{R}_1 \cap \mathscr{R}_2$. Then $\left\langle \mathscr{R}_1\mathcal{S} \right\rangle \cap \mathscr{R}_2 = \left\langle (\mathscr{R}_1 \cap \mathscr{R}_2)\mathcal{S} \right\rangle$.
}
\end{cor}
\begin{proof}
Let $A, B$ be operators in $\mathcal{S}$. From Proposition \ref{prop:singly-gen}, \ref{cor:intersection}, we have that $\left\langle \mathscr{R}_1 \{A, B\} \right\rangle \cap \mathscr{R}_2 = (\mathscr{R}_1A + \mathscr{R}_1B) \cap \mathscr{R}_2 = \mathscr{R}_1 \sqrt{A^*A + B^*B} \cap \mathscr{R}_2 = (\mathscr{R}_1 \cap \mathscr{R}_2)\sqrt{A^*A + B^*B} = (\mathscr{R}_1 \cap \mathscr{R}_2) A + (\mathscr{R}_1 \cap \mathscr{R}_2) B = \left\langle (\mathscr{R}_1 \cap \mathscr{R}_2) \{A, B\} \right\rangle$. Thus $\left\langle \mathscr{R}_1\mathcal{S} \right\rangle \cap \mathscr{R}_2 = \left\langle (\mathscr{R}_1 \cap \mathscr{R}_2)\mathcal{S} \right\rangle$.
\end{proof}

The corollary below is in the same vein as Proposition \ref{prop: ideal}. In effect, it says that every left ideal of a represented von Neumann algebra may be viewed as the intersection of a left ideal of the full algebra of bounded operators on the underlying Hilbert space with the von Neumann algebra.

\begin{cor}
\textsl{
Let $\mathscr{R}$ be a von Neumann algebra acting on the Hilbert space $\mathscr{H}$. Let $\mathscr{I}$ be a left ideal of $\mathscr{R}$. Then there is a left ideal $\mathscr{J}$ of $\mathcal{B}(\mathscr{H})$ such that $\mathscr{I} = \mathscr{J} \cap \mathscr{R}$.
}
\end{cor}
\begin{proof}
By choosing $\mathscr{R}_1 = \mathcal{B}(\mathscr{H}), \mathscr{R}_2 = \mathscr{R}$ and $\mathcal{S} = \mathscr{I}$ and using Corollary \ref{cor:intersection}, we see that for $\mathscr{J} := \left\langle \mathcal{B}(\mathscr{H})\mathscr{I} \right\rangle$, we have that, $\mathscr{I} = \left\langle \mathscr{R} \mathscr{I} \right\rangle = \left\langle (\mathcal{B}(\mathscr{H}) \cap \mathscr{R}) \mathscr{I} \right\rangle = \mathscr{J} \cap \mathscr{R}$ and $\mathscr{J}$ is a left ideal of $\mathcal{B}(\mathscr{H})$.
\end{proof}

\subsection{Results on $C^*$-$2$-convex sets}

In this subsection, we prove that $C^*$-$2$-convex subsets of bimodules over a von Neumann algebra are $C^*$-convex.

\begin{lemma}
\label{lem:prep-cstar}
\textsl{
Let $\mathscr{R}$ be a von Neumann algebra acting on the Hilbert space $\mathscr{H}$, with identity $I.$ If $T_1, \cdots, T_n$ are operators in $\mathscr{R}$ such that $T_1^* T_1 + \cdots + T_n^* T_n = I$, then there are operators $S_1, \cdots, S_{n-1}$ in $\mathscr{R}$ such that $T_i = S_i \sqrt{I-T_n^* T_n}, 1 \le i \le n-1$ and $S_1^* S_1 + \cdots S_{n-1}^* S_{n-1}$ is the range projection of $\sqrt{I - T_n^*T_n}$.
}
\end{lemma}
\begin{proof}
As $T_1^* T_1 + \cdots + T_{n-1}^* T_{n-1} = I - T_n^* T_n$, we have that $$T_i^* T_i \le \sqrt{I - T_n^* T_n}\sqrt{I - T_n^* T_n}, 1 \le i \le n-1.$$ Since $\sqrt{I - T_n^* T_n}$ is self-adjoint, the orthogonal complement of the range of $\sqrt{I - T_n^* T_n}$ is equal to the kernel of $\sqrt{I - T_n^* T_n}$. From the proof of the Douglas lemma for von Neumann algebras in Theorem \ref{theorem:douglas}, we have for each $i \in \{1, 2, \cdots, n-1 \}$, an operator $S_i$ in $\mathscr{R}$ such that $T_i = S_i \sqrt{I - T_n^* T_n}$, and  $S_i z = 0$ for any vector $z$ in $\mathrm{ker}(\sqrt{I - T_n^* T_n})$. Note that, $$I - T_n^* T_n = T_1^* T_1 + \cdots + T_{n-1}^* T_{n-1} = \sqrt{I - T_n^* T_n}(S_1^* S_1 + \cdots + S_{n-1}^* S_{n-1})\sqrt{I - T_n^* T_n}.$$
As a result, for every vector $x$ in the range of $\sqrt{I - T_n^* T_n}$, we have $\langle x, x \rangle = \langle (S_1^* S_1 + \cdots + S_{n-1}^* S_{n-1})x, x \rangle.$ In addition, for every vector $g$ in the kernel of $\sqrt{I - T_n^* T_n}$, we have $\langle (S_1^* S_1 + \cdots + S_{n-1}^* S_{n-1})y, y \rangle = 0$. Thus  the operator $S_1^* S_1 + \cdots + S_{n-1}^* S_{n-1}$ must be the range projection of $\sqrt{I - T_n^* T_n}$.
\end{proof}

\begin{prop}
{\sl
\label{prop:segment}
For a {\it finite} von Neumann algebra $\mathscr{R}$ and a Hilbert $\mathscr{R}$-bimodule $\mathfrak{H}$, a subset $\mathscr{S}$ of $\mathfrak{H}$ is $\mathscr{R}$-convex if and only if  $\mathscr{S}$ is $\mathscr{R}$-$2$-convex.}
\end{prop}

\begin{proof}
If $\mathscr{S}$ is $\mathscr{R}$-convex, the $\mathscr{R}$-segment $S(A_1, A_2)$ is clearly in $\mathscr{S}$ for any $A_1, A_2 \in \mathscr{S}$ as it consists of $\mathscr{R}$-convex combinations of $A_1$ and $A_2$. For the other direction, we inductively prove that for $(A_1, \cdots, A_n)$, an $n$-tuple of vectors from $\mathscr{S}$ and $T_1, \cdots,T_n \in \mathscr{R}$ satisfying $T_1^* T_1 + \cdots T_n^* T_n = I$, the $\mathscr{R}$-convex combination $T_1^*A_1T_1 + \cdots T_n^* A_n T_n$ is in $\mathscr{S}$. For $n = 1, 2$, the above is clearly true from the $\mathscr{R}$-$2$-convexity of $\mathscr{S}$.

As $\mathscr{R}$ is finite, by \cite[Exercise 6.9.10(ii)]{kadison-ringrose3}, each of the $T_i$'s has a unitary polar decomposition, that is, there are unitary operators $U_i$ and positive operators $P_i$ such that $T_i = U_i P_i$, for $1 \le i \le n$. The vectors $A_i' := U_i^*A_iU_i$ are in $\mathscr{S}$ and $P_1^2 + \cdots + P_n^2 = I$. From Lemma \ref{lem:prep-cstar}, there are operators $S_1, \cdots, S_{n-1}$ in $\mathscr{R}$ as defined in Lemma \ref{lem:prep-cstar} such that $P_i = S_i \sqrt{I - P_n^2}, 1 \le i \le n-1$ and $S_1^*S_1 + \cdots + S_{n-1}^*S_{n-1} = E$  where $E$ is the range projection of $\sqrt{I - P_n^2}$. Let $F$ denote the projection onto the kernel of $\sqrt{I - P_n^2}$. As $F = I-E$, clearly $F$ is in $\mathscr{R}$. For $i \in \{ 1, \cdots, n-1 \}$ as $ \mathrm{ker}(\sqrt{I-P_n^2}) \subseteq \mathrm{ker}(P_i)$, we have that $P_i F = F P_i =0$ and as a result $FS_i = 0$. Now define $S_i ' := S_i + \frac{F}{\sqrt{n-1}}, 1 \le i \le n-1$. We see that $S_1 '^*S_1' + \cdots + S_{n-1}'^*S_{n-1}' = (S_1^*S_1 + \frac{F}{n-1}) + \cdots + (S_{n-1}^*S_{n-1} + \frac{F}{n-1}) = E + F = I$ and $P_i = S_i'\sqrt{I - P_n^2}$. By the induction hypothesis, $A' := S_1'^* A_1'S_1 + \cdots S_{n-1}'^* A_{n-1}'S_{n-1}'$ is in $\mathscr{S}$. As $T_1^*A_1T_1 + \cdots T_n^* A_n T_n = \sqrt{I - P_n^2}A'\sqrt{I - P_n^2} + P_n A_n' P_n$,  being a $\mathscr{R}$-convex combination of $A'$ and $A_n'$, it must be in $\mathscr{S}$. 
\end{proof}

At this point, we refer the reader to Chapter 6 of \cite{kadison-ringrose2} for a detailed account of the comparison theory of projections in von Neumann algebras. We denote the Murray-von Neumann equivalence relation for projections in a von Neumann algebra by $\sim$ and the partial order it begets by $\precsim$. Below we mention (without proof) the halving lemma (see \cite[Lemma 6.3.3]{kadison-ringrose2}) for properly infinite projections in a von Neumann algebra as it will be repeatedly used in Proposition \ref{prop:segment2}.

\begin{lemma}[Halving Lemma]
\label{lem:halving}
{\sl
Let $E$ be a properly infinite projection in a von Neumann algebra $\mathscr{R}$. There is a projection $F$ in $\mathscr{R}$ such that $F \le E$ and $F \sim (E - F) \sim E$.}
\end{lemma}

\begin{prop}
\label{prop:segment2}
{\sl For an infinite von Neumann algebra $\mathscr{R}$ and a Hilbert $\mathscr{R}$-bimodule $\mathfrak{H}$, a subset $\mathscr{S}$ of $\mathfrak{H}$ is $\mathscr{R}$-convex if and only if  $\mathscr{S}$ is $\mathscr{R}$-$2$-convex.}
\end{prop}
\begin{proof}
It is straightforward from the definitions that every $\mathscr{R}$-convex subset is $\mathscr{R}$-$2$-convex. We prove the other direction inductively. For $n \in \mathbb{N}$, let the $\mathscr{R}$-polytope generated by any $(n-1)$-tuple of elements from $\mathscr{S}$ be contained in $\mathscr{S}$. Below we prove that the $\mathscr{R}$-polytope generated by an $n$-tuple $(A_1, A_2, \cdots, A_n)$ of elements from $\mathscr{S}$ is contained in $\mathscr{S}$. In other words, for operators  $T_1, \cdots, T_n$ in $\mathscr{R}$ such that $\sum_{i=1}^n T_i ^* T_i = I$, we prove that $\sum_{i=1}^n T_i ^* A_i T_i$ is in $\mathscr{S}$. Note that for $n = 1, 2$, the above is clearly true from the hypothesis of $\mathscr{R}$-$2$-convexity of $\mathscr{S}$.

Repeatedly using Lemma \ref{lem:halving}, consider mutually orthogonal projections $E_1, E_2, \ldots, E_n$ and projections $E_{n1}, E_{n2}, \ldots, E_{nn}$ in $\mathscr{R}$ such that 
\begin{align*}
&E_1 + E_2 + \cdots + E_n = I,\\
&E_n = E_{n1} + E_{n2} + \cdots + E_{nn},\\
&E_1 \sim  E_2 \sim \cdots \sim  E_n \sim E_{n1} \sim E_{n2} \sim \cdots \sim  E_{nn} \sim I.
\end{align*}
For $i \in \{ 1, 2, \ldots, n-1 \}$, define $F_i:= E_i + E_{ni}$. s $I \sim E_i \le F_i \le I$, from the reflexivity of $\precsim$ we have that $E_i \sim F_i \sim I$. For $i \in \{1, 2, \ldots , n-1 \}$, let $W_i$ be a partial isometry with initial projection $F_i$ and final projection $I$, and $W_i '$ be a partial isometry with initial projection $E_i$ and final projection $F_i$. Define $W := \sum_{i=1}^{n-1} W_i '$. Note that $W$ itself is a partial isometry with initial projection $\sum_{i=1}^{n - 1} E_i = I - E_n$ and final projection $\sum_{i=1}^{n-1} F_i = I$. Further let $W_n$ be a partial isometry with initial projection $E_n$ and final projection $I$. For $i \in \{ 1, \ldots, n-1 \}$, let $V_i := W_i W$ and define $V_n := W_n$. To assist the reader in navigating the maze of partial isometries we have defined, we tabulate the partial isometries and their initial and final projections in Table \ref{tab:partial}. Recall that for a partial isometry $V$ with initial projection $E$ and final projection $F$, we have that $V^*V = E$ and $VV^* = F$.
\begin{table}[ht]
\centering
\begin{tabular}{|c|c|c|}
\hline
Partial Isometry & Initial Projection & Final Projection \\
\hline
$W_i$ & $F_i$ & $I$\\
$W_i ' $ & $E_i$ & $F_i$ \\
$W ( = \sum_{j=1}^{n-1} W_j ')$ & $I-E_n$ & $I$\\
$W_n$ & $E_n$ & $I$\\
$V_i$ & $E_i$ & $I$\\
\hline
\end{tabular}
\caption{Reference table for the partial isometries. The index $i$ ranges from $1$ to $n-1$.}
\label{tab:partial}
\end{table}

By the induction hypothesis, the vector $A := W_1 ^* A_1 W_1 + \cdots + W_{n-1} ^* A_{n-1} W_{n-1}$ is in 
$\mathscr{S}$ as $\sum_{i=1}^{n-1} W_i ^* W_i = \sum_{i=1}^{n-1} F_i = I$. Note that $W^*AW = \sum_{i=1}^{n-1} V_i ^* A_i V_i$  and the vector $W^*AW + W_n^* A_n W_n (= \sum_{i=1}^{n} V_i ^* A_i V_i)$ is in the $\mathscr{R}$-segment joining $A$ and $A_n$ as $W^*W + W_n ^* W_n = (I-E_n) + E_n = I$, and thus $\sum_{i=1}^{n} V_i ^* A_i V_i$ is in $\mathscr{S}$. We further have that $\sum_{i=1}^n V_i ^* V_i = (\sum_{i=1}^{n-1} W^*W_i^*W_i W ) +W_n ^*
W_n = (\sum_{i=1}^{n-1} W^*F_i W) + E_n =  (\sum_{i=1}^{n-1} E_i) + E_n  = I$. Consider the operator $\widetilde{V} := V_1^*T_1 + \cdots + V_n ^* T_n$ in $\mathscr{R}$ and the vector $\widetilde{A} := V_1 ^*A_1 V_1 + \cdots + V_n ^* A_n V_n$ in $\mathfrak{H}$. As $V_i V_j^* = \delta _{ij} I$ for $1 \le i, j \le n$, note that $\widetilde{V}^*\widetilde{V} = \sum_{i=1}^n \sum_{j=1}^n T_i ^* V_i V_j ^* T_j = \sum_{i=1}^n T_i ^* V_i V_i ^* T_i = \sum_{i=1}^n T_i ^* T_i = I$ and $\widetilde{V}^*\widetilde{A}\widetilde{V} = \sum_{i=1}^n \sum_{j=1}^n \sum_{k=1}^n T_i ^* V_i V_j ^* A_j V_j V_k ^* T_k = \sum_{i=1}^n T_i ^* A_i T_i$. We have already proved that $\widetilde{A}$ is in $\mathscr{S}$. Thus we have that $\widetilde{V}^*\widetilde{A}\widetilde{V} = \sum_{i=1}^n T_i ^* A_i T_i$ is in $\mathscr{S}$. This finishes the proof.
\end{proof}

From the type decomposition of von Neumann algebras (see \cite[Theorem 6.5.2.]{kadison-ringrose2}), for a von Neumann algebra $\mathscr{R}$, we have central projections $E, F$ in $\mathscr{R}$ such that $E + F = I$, $\mathscr{R}E$ is a finite von Neumann algebra acting on $E(\mathscr{H})$, and $\mathscr{R}F$ is a properly infinite von Neumann algebra acting on $F(\mathscr{H})$. Thus combining Proposition \ref{prop:segment} and Proposition \ref{prop:segment2}, we have the following theorem.

\begin{thm}
\label{thm:segment}
{\sl For a von Neumann algebra $\mathscr{R}$ and an $\mathscr{R}$-bimodule $\mathfrak{H}$, a subset $\mathscr{S}$ of $\mathfrak{H}$ is $\mathscr{R}$-convex if and only if $\mathscr{S}$ is $\mathscr{R}$-$2$-convex.}
\end{thm}

\bibliographystyle{plain}
\bibliography{reference}

\end{document}